\documentclass[letter,10pt]{amsart}
\usepackage{amsmath, amsfonts, amssymb,amscd,graphics,graphicx,color,eucal,setspace} 
\usepackage{latexsym}   
\usepackage{color}

\usepackage{url}

\numberwithin{equation}{section}

\theoremstyle{plain}
\newtheorem{thm}{Theorem}[section]
\newtheorem{lem}[thm]{Lemma} 

\newtheorem{prop}[thm]{Proposition} 

\theoremstyle{definition}
\newtheorem*{rmk}{Remark}

\newtheorem*{defn}{Definition}

\def\QC{\mathbb C}

\def\S1{S}

\def\q{r}

\def\f{\theta}

\def\xi{p}

\begin{document}
  
\title[Equivariant cohomology rings of peterson varieties]{The equivariant cohomology rings of peterson varieties in all Lie types}
\author {Megumi Harada} 
\address{Department of Mathematics and Statistics, McMaster University, 1280 Main Street West, Hamilton, Ontario L8S4K1, Canada}
\email{Megumi.Harada@math.mcmaster.ca}
\urladdr{\url{http://www.math.mcmaster.ca/~haradam}}

\author {Tatsuya Horiguchi}
\address{Department of Mathematics, Osaka City University, Sumiyoshi-ku, Osaka 558-8585, Japan}
\email{d13saR0z06@ex.media.osaka-cu.ac.jp}

\author {Mikiya Masuda}
\address{Department of Mathematics, Osaka City University, Sumiyoshi-ku, Osaka 558-8585, Japan}
\email{masuda@sci.osaka-cu.ac.jp}
\date{\today}

\keywords{equivariant cohomology, Peterson varieties, flag varieties, Monk formula, Giambelli formula} 

\subjclass[2000]{Primary: 55N91, Secondary: 14N15} 

\begin{abstract}
Let $G$ be a complex semisimple linear algebraic group and let $Pet$ be the associated \textbf{Peterson variety} in the flag variety $G/B$. 
The main theorem of this note gives an efficient presentation of the equivariant cohomology ring $H^*_S(Pet)$ of the 
Peterson variety as a quotient of a polynomial ring by an ideal $J$ generated by quadratic polynomials, in the spirit of the Borel presentation of the cohomology of the flag variety. Here the group $S \cong \mathbb{C}^*$ is a certain subgroup of a maximal torus $T$ of $G$. 
Our description of the ideal $J$ uses the Cartan matrix and is uniform across Lie types. In our arguments we use the Monk formula and Giambelli formula for the equivariant cohomology rings of Peterson varieties for all Lie types, as obtained in the work of Drellich. Our result generalizes a previous theorem of Fukukawa-Harada-Masuda, which was only for Lie type $A$. 
\end{abstract}

\maketitle

\setcounter{tocdepth}{1}
\tableofcontents

\section{Introduction} \label{sect:1}

The main result of this paper is an explicit and efficient
presentation of the equivariant cohomology ring
of Peterson
varieties in all Lie types 
in terms of generators and relations, in the spirit of the
well-known Borel presentation of the cohomology of the flag
variety. 
We briefly recall the setting of our results. 
The \textbf{Peterson variety} has been much studied due to its relation, for example, to 
the quantum cohomology of the flag variety \cite{Kos96,
  Rie03}. Thus it is natural to study their topology, e.g. the
structure of their (equivariant) cohomology rings. 
In Lie type $A$ the Peterson variety $Pet$ can be easily described in a concrete manner. Specifically, it is defined to be the 
following subvariety of the full flag variety
$\mathcal{F}\ell ags(\mathbb{C}^n)$:
\begin{equation}\label{eq:def intro}
Pet := \{ V_\bullet \, \mid \,  NV_i \subseteq
V_{i+1} \textup{ for all } i = 1, \ldots, n-1\}
\end{equation}
where $V_\bullet$ denotes a nested sequence 
$0 \subseteq V_1 \subseteq V_2 \subseteq \cdots \subseteq V_{n-1} \subseteq V_n = \mathbb{C}^n$ of subspaces of $\mathbb{C}^n$ and $\dim_\mathbb{C} V_i = i$ for all $i$ and 
$N: \mathbb{C}^n \to \mathbb{C}^n$ denotes a principal nilpotent operator.
For a general complex semisimple linear algebraic group $G$, there is also a concrete Lie-theoretic description of the Peterson variety as a subvariety of its flag variety $G/B$,
which we recall in Section~\ref{sect:2} below. 

In addition, there is a natural subgroup $S \cong \mathbb{C}^*$ (which is a subgroup of the maximal torus $T$ of $G$) which acts 
on $Pet$ (for details see Section~\ref{sect:2}). 
The inclusion of $Pet$
into $G/B$ induces a natural ring homomorphism
\begin{equation}\label{eq:intro-proj}
H^*_T(G/B) \to H^*_{S}(Pet). 
\end{equation} 
The purpose of this manuscript is to give
an efficient presentation of the equivariant cohomology ring
$H^*_{S}(Pet)$ in all Lie types, using the recent work of Drellich \cite{d} which gives a Monk formula and a Giambelli formula for Peterson varieties in all Lie types. In particular, we are able to obtain a uniform description of the relevant ideal $J$, valid for all Lie types, using the Cartan matrix associated to the Lie algebra $\mathfrak{g}$ of $G$. In particular, our analysis shows that the ideal $J$ is generated by quadratics.  

Our proof uses Hilbert series and regular sequences, in a similar spirit to previous work of Fukukawa and the first and third authors \cite{f-h-m}, which was in turn motivated by the work of \cite{FukukawaIshidaMasuda-typeA, Fukukawa-G2} which computes the graph cohomology of the GKM graphs of the flag varieties of classical type and of $G_2$. 

Brion and Carrell give a different description of the equivariant cohomology rings of Peterson (and other nilpotent Hessenberg) varieties as the coordinate rings of certain affine curves \cite[Theorem 3]{Brion-Carrell} using techniques from algebraic geometry. 
In a related direction, by extending the ideas of the present manuscript, we give in \cite{AbeHaradaHoriguchiMasuda} a uniform description (with explicit generators and relations) of the equivariant cohomology rings of arbitrary regular nilpotent Hessenberg varieties  -- a class of varieties which includes the Peterson variety and the full flag variety -- in Lie type A. 

This paper is organized as follows. We briefly recall the necessary
background in Section~\ref{sect:2}. We derive the relevant quadratic relations in Section~\ref{sect:3}. In particular, a key computation is contained in Lemma~\ref{lemcartan}. The main theorem is proven in Section~\ref{sect:4}.

\medskip
\noindent
\textbf{Acknowledgements.} The first author is supported in part by an NSERC Discovery Grant (Individual), an Ontario Early Researcher Award, a Canada Research Chair (Tier 2) Award, and a Japan Society for the Promotion of Science Invitation Fellowship for Research in Japan (Fellowship ID L-13517). The first author additionally thanks the Osaka City University Advanced Mathematics Institute for its hospitality, which made this collaboration possible. The third author is partially supported by a JSPS Grant-in-Aid for Scientific Research 25400095.

\section{Background on the Peterson variety} \label{sect:2}

In this section we record some facts about Peterson varieties which we require in this manuscript. 

Let $G$ be a complex semisimple linear algebraic group of rank $n$. 
We fix $B$ a Borel subgroup and $T$ a maximal torus of $G$ such that $T \subseteq B \subseteq G$. These choices then determine 
the following data: 

\begin{itemize} 




\item a set of simple roots $\Delta=\{\alpha_1,\dots,\alpha_n\}$,

\item the associated Weyl group $W$, 

\item the associated Lie algebras $\mathfrak{t}\subseteq\mathfrak{b}\subseteq\mathfrak{g}$, and 

\item root spaces $\mathfrak{g}_{\alpha} \subseteq \mathfrak{g}$ for each root $\alpha$.

\end{itemize}

\begin{defn} \label{defpet}
Let $E_{\alpha}$ be a basis element of the root space $\mathfrak{g}_{\alpha}$ and let $N_0=\sum_{\alpha\in\Delta}E_{\alpha }$, a regular nilpotent operator. In this setting we may define the \textbf{Peterson variety (associated to $\mathfrak{g}$)} as 
\begin{equation*} \label{eqS}
Pet:=\{gB\in G/B \mid Ad(g^{-1})(N_0)\in \mathfrak{b}\oplus\displaystyle\bigoplus_{\alpha\in-\Delta}\mathfrak{g}_{\alpha}\}. 
\end{equation*}
\end{defn}

As is well-known, the maximal torus $T$ acts on $G/B$ by left multiplication. This action does not in general preserve 
the Peterson variety. However, using the homomorphism 
$\phi:T\rightarrow(\mathbb{C}^*)^n$ defined by $t\mapsto (\alpha_1(t),\dots,\alpha_n(t))$
and defining $S$ to be the connected component of the identity in
\begin{equation*} 
\phi^{-1}(\{(c,c,\dots,c) \mid c\in\mathbb{C}^*\})
\end{equation*}
it can be seen that the restriction of the $T$-action on $G/B$ to the subgroup $S$ does preserve $Pet$ 
(\cite[Lemma 5.1]{h-t}). 

Next recall that the $T$-fixed points of $G/B$ are in bijective correspondence with the Weyl group $W$ of $G$. Moreover, since the 
$S$-fixed points $Pet^S$ of the Peterson variety satisfy the relation
\[
Pet^S=Pet\cap (G/B)^T
\]
we may view $Pet^S$ as a subset of the Weyl group $W$. 
Indeed, the fixed point set $Pet^S$ may be described concretely as follows. For 
a subset $K$ of the set $\Delta$ simple roots, let $W_K$ denote the parabolic subgroup generated by $K$ and let $w_K$ denote 
the longest element of $W_K$. Then it is known \cite[Proposition 5.8]{h-t} that 
\begin{equation*}
Pet^S=\{w_K \mid K\subseteq\Delta \}.
\end{equation*}

Here and below we always use complex coefficients $\mathbb{C}$ for our cohomology rings and hence omit it from our notation. 
Let $\alpha_i: T\rightarrow \mathbb{C}^*$ be a homomorphism which thus determines 
a complex $1$-dimensional representation of $T$. Let $ET \times_T \mathbb{C} \rightarrow BT$ be the corresponding complex line bundle 
and by slight abuse of notation we let $\alpha_i \in H^2(BT)$ also denote the corresponding first Chern class. With this notation in place 
we have 
\[
H^*(BT)=\mathbb{C}[\alpha_1,\dots,\alpha_n]. 
\]
Consider the $1$-dimensional representation of the diagonal subgroup $\{(c,c,\ldots, c): c \in \mathbb{C}^*\} \subseteq (\mathbb{C}^*)^n$
obtained via the projection 
$(c,c,\dots,c)\to c$. Composing with the restriction to $S$ of the above homomorphism $\phi$, we obtain
a $1$-dimensional representation of $S$ and an associated line bundle $ES \times_S \mathbb{C} \to BS$ with first Chern class denoted $t \in H^2(BS)$. With this notation in place we have 
\[
H^*(BS)=\mathbb{C}[t].
\]
Next we recall that the inclusion homomorphism $S \hookrightarrow T$ induces a homomorphism 
$\pi\colon H^*(BT)\to H^*(BS)$ and from the definition of $\phi$ we obtain 
\begin{equation}\label{eq:2-1}
\pi(\alpha_i)=t\quad (i=1,2,\dots,n). 
\end{equation}

We now consider the following commutative diagram
\begin{equation}\label{eq:2-2}
\begin{CD}
H^{\ast}_{T}(G/B)@>
>> \displaystyle{\bigoplus_{w\in (G/B)^T=W} H^{\ast}_{T}(w)}\\
@V{\rho}VV @V{\pi}VV\\
H^{\ast}_{S}(Pet)@>
>> \displaystyle{\bigoplus_{w\in Pet^S\subseteq W} H^{\ast}_{S}(w)}
\end{CD}
\end{equation}
where all the maps are induced from inclusions of subgroups or inclusions of subspaces. As is well-known, 
the odd cohomology $H^{odd}(G/B)$ of $G/B$ vanishes. The same holds for the Peterson variety, i.e. $H^{odd}(Pet) = 0$
\cite{pr13}. Thus we obtain that both horizontal maps in~\eqref{eq:2-2} are injective, and 
we may identify $H^{\ast}_{T}(G/B)$ (respectively $H^{\ast}_{S}(Pet)$) with its image under these maps.
For $w \in (G/B)^T \cong W$ (respectively $w \in Pet^S \subseteq W$) and 
$f\in H^{\ast}_{T}(G/B)$ (respectively $f\in H^{\ast}_S(Pet)$)
we will denote by $f(w)$ the 
restriction of $f$ to the $w$-th factor 
$H_T^{\ast}(w)=H^{\ast}(BT)=\mathbb{C}[\alpha_1,\dots,\alpha_n]$ (resp. $H^{\ast}_S(w)=H^{\ast}(BS)=\mathbb{C}[t]$)
in the direct products on the right hand sides of~\eqref{eq:2-2}.

For $v \in W$, we let $\sigma_v$ denote the corresponding equivariant Schubert class in $H^{\ast}_{T}(G/B)$, and let $p_v$ denote its image $\rho(\sigma_v)$ in $H^{\ast}_{S}(Pet)$. We call $p_v$ a \textbf{Peterson Schubert class} (associated to $v$). 
Let $s_i$ be the simple reflection corresponding to a simple root $\alpha_i$. 
The vertices of the Dynkin diagram corresponding to the set of simple roots 
$\Delta=\{\alpha_1,\dots,\alpha_n\}$ is in 1-1 correspondence 
with $\Delta$. Here and below, we assume to be fixed an ordering of the simple roots 
as given in \cite[Figure 1]{d} (which in turn agrees with the standard ordering in \cite[p.58]{hu80}). 
With respect to this ordering, given 
any subset $K=\{\alpha_{a_1},\alpha_{a_2},\dots,\alpha_{a_k} \}$ of the 
simple roots with $a_1<a_2<\dots<a_k$, we define an element $v_K$ of $W$ by the formula 
\begin{equation*}
v_K:=s_{a_1}s_{a_2}\cdots s_{a_k}. 
\end{equation*}
The Peterson Schubert classes $p_{v_K}$ corresponding to the Weyl group elements $v_K$ defined above satisfy the following property. 

\begin{prop}[Theorem 3.5 in \cite{d}] \label{thmpvk}
The Peterson Schubert classes $\{p_{v_K} \mid K\subseteq\Delta \}$ form a $\mathbb{C}[t]$-module basis for $H^{\ast}_{S}(Pet)$. 
\end{prop}

It follows from Proposition~\ref{thmpvk} that the $\rho$ in \eqref{eq:2-2} is surjective. 
It is well-known that the equivariant Schubert classes $\sigma_{s_i}$ generate
$H^{\ast}_{T}(G/B)$ as a 
$\mathbb{C}[\alpha_1,\dots,\alpha_n]$-algebra. From the surjectivity of the homomorphism $\rho$ we immediately obtain the following. 

\begin{prop} \label{Propgene} 
The Peterson Schubert classes $p_{s_i}$ $(i=1,2,\dots,n)$
generate $H^{\ast}_{S}(Pet)$ as a $\mathbb{C}[t]$-algebra.  
\end{prop}

Since the odd cohomology $H^{odd}(Pet)$ of the Peterson variety vanishes, we know that as a $\mathbb{C}[t]$-module the equivariant cohomology $H^*_S(Pet)$ is isomorphic to $\mathbb{C}[t]\otimes H^*(Pet)$. 
It is known \cite[Theorem 3]{Brion-Carrell} that 
\begin{equation} \label{eq:2-3}
\begin{split}
F(H^*(Pet),s)&=(1+s^2)^n\\
F(H^*_S(Pet),s)&=\frac{(1+s^2)^n}{1-s^2}
\end{split}
\end{equation}
where the left hand sides denotes the Hilbert series of the graded rings $H^{\ast}(Pet)$ and $H^*_S(Pet)$ with respect to the variable $s$ (of degree $1$).

Fix an integer $i$ with $1 \leq i \leq n$ and a subset $K \subseteq \Delta$.
From 
Proposition~\ref{thmpvk} it follows that the product
$p_{s_i}\cdot p_{v_K}$ can be written uniquely as a $\mathbb{C}[t]$-linear 
combination of the $p_{v_J}$ (for $J\subseteq\Delta$). 
The so-called Monk's formula gives a concrete computation of the coefficients in this linear combination. 

\begin{thm}[Monk's formula for Peterson varieties for all Lie types, Theorem 4.2 in \cite{d}] \label{thmMonk} 
The Peterson Schubert classes satisfy the following relation: 
\begin{equation*}\label{eqMonk}
p_{s_i}\cdot p_{v_K}=p_{s_i}(w_K)\cdot p_{v_K}+\displaystyle\sum_{\substack{J\supset K \\ |J|=|K|+1}} c_{i,K}^J\cdot p_{v_J}
\end{equation*}
where the coefficient $c_{i,K}^J$ are non-negative rational numbers. More specifically, we have 
\begin{equation*}
c_{i,K}^J=(p_{s_i}(w_J)-p_{s_i}(w_K))\cdot \frac {p_{v_K}(w_J)}{p_{v_J}(w_J)}. 
\end{equation*}
\end{thm}

Next we recall the so-called 
Giambelli's formula. From  Proposition~\ref{Propgene} it follows that each module generator $p_{v_K}$ can be expressed as a polynomial (with $\mathbb{C}[t]$ coefficients) in the (ring) generators $p_{s_i}$. The Giambelli formula gives a concrete expression for this polynomial as follows. 

\begin{thm}[Giambelli's formula for Peterson varieties for all Lie types,  Theorem 5.5 in \cite{d}] \label{thmGiambelli} 
Suppose $K$ is a subset of the simple roots $\Delta$. Assume that the Dynkin diagram corresponding to the subset $K$ is connected. Then 
\begin{equation*}\label{eqGiambelli}
\frac{|K|!}{|\mathcal R(v_K)|} \cdot p_{v_K}=\displaystyle\prod_{\alpha_i\in K} p_{s_i} 
\end{equation*}
where $|\mathcal R(v_K)|$ denotes the number of distinct reduced-word expressions for $v_K$. 
\end{thm}

\begin{rmk}[cf. Theorem 5.3 in \cite{d}]
The connectedness assumption in the above theorem is not serious, in the following sense. Suppose $J,K\subseteq\Delta$ are two subsets of $\Delta$ such that their corresponding Dynkin diagrams are connected. Suppose, however, that $J \cup K$ has corresponding Dynkin diagram that is not connected. Then $p_{v_{J \cup K}}$ is simply the product of $p_{v_J}$ and $p_{v_K}$, i.e. 
\begin{equation*}
p_{v_{J\cup K}}=p_{v_J}\cdot p_{v_K}. 
\end{equation*}
\end{rmk}

\section{Quadratic relations satisfied by the Peterson Schubert classes $p_{s_i}$}\label{sect:3}

In this section, we derive certain quadratic relations satisfied by the cohomology-degree-$2$ Peterson Schubert classes $p_{s_i}$ 
$(i=1,2,\dots,n)$ by using Monk's formula (Theorem~\ref{thmMonk}), Giambelli's formula (Theorem~\ref{thmGiambelli}), and 
Billey's formula recalled below. We will then show in the next section that these relations are sufficient to determine the equivariant cohomology ring $H^{\ast}_{S}(Pet)$ of the Peterson variety. 

\begin{thm}[Billey's formula, Theorem 4 in \cite{b}] \label{thmBilley} 
Let 
$w\in W$ 
and fix a reduced word decomposition $w=s_{b_1}s_{b_2}\cdots s_{b_m}$ of $w$.  
Set $r(i,w) :=s_{b_1}s_{b_2}\cdots s_{b_{i-1}}(\alpha_{b_i})$.  For an equivariant Schubert class $\sigma_v$ for $v \in W$ we have the following: 
\begin{equation*} 
\sigma_v(w)=\displaystyle\sum_{\substack{reduced \ words \\ v=s_{b_{j_1}}s_{b_{j_2}}\cdots s_{b_{j_\ell}}}}\displaystyle\prod_{i=1}^{\ell} r(j_i,w).
\end{equation*}
\end{thm}

We begin with some elementary computations involving Peterson Schubert classes. First, from 
Monk's formula (Theorem~\ref{thmMonk}) applied to the case $K=\{\alpha_i\}$ and $v_K=s_i$ we obtain 
\begin{equation} \label{eq:3-1}
p_{s_i}^2=p_{s_i}(s_i)\cdot p_{s_i}+\displaystyle\sum_{j\neq i} c_i^j\cdot p_{v_{\{\alpha_i,\alpha_j\}}} 
\end{equation}
where 
\begin{equation} \label{eq:3-2}
c_i^j=(p_{s_i}(w_{\{\alpha_i,\alpha_j\}})-p_{s_i}(s_i))\cdot \frac {p_{s_i}(w_{\{\alpha_i,\alpha_j\}})}{p_{v_{\{\alpha_i,\alpha_j\}}}(w_{\{\alpha_i,\alpha_j\}})}. 
\end{equation}
More specifically, since Theorem~\ref{thmBilley} implies that  $\sigma_{s_i}(s_i)=\alpha_i$, from \eqref{eq:2-1} we conclude 
\begin{equation} \label{eq:3-3}
p_{s_i}(s_i)=t. 
\end{equation}

We record the following. 
\begin{lem} \label{lemcommute}
In~\eqref{eq:3-1}, if $s_i$ and $s_j$ commute, then $c_i^j=0$.
\end{lem}

\begin{proof}
Since 
$s_i$ and $s_j$ commute, we have $w_{\{\alpha_i,\alpha_j\}}=s_is_j$. Moreover  
from Theorem~\ref{thmBilley} we can compute 
\[
\sigma_{s_i}(w_{\{\alpha_i,\alpha_j\}})=\sigma_{s_i}(s_is_j)=\alpha_i. 
\]
From~\eqref{eq:2-1} we get $p_{s_i}(w_{\{\alpha_i,\alpha_j\}})=t$.  Then the equations~\eqref{eq:3-2}, \eqref{eq:3-3} imply $c_i^j=0$ as desired. 
\end{proof}

In the case when $s_i$ and $s_j$ do not commute, the Dynkin diagram corresponding to the subset 
$K=\{\alpha_i,\alpha_j\}$ is connected, so Giambelli's formula (Theorem~\ref{thmGiambelli}) yields
\begin{equation} \label{eq:3-4}
p_{v_{\{\alpha_i,\alpha_j\}}}=\frac{1}{2}p_{s_i}p_{s_j}.
\end{equation}

In this case, the coefficient appearing in~\eqref{eq:3-1} can be expressed in terms of the Cartan matrix. 

\begin{lem}\label{lemcartan}
In~\eqref{eq:3-1}, if $s_i$ and $s_j$ do not commute, then  
\begin{equation*}
c_i^j=-\langle\alpha_i,\alpha_j\rangle
\end{equation*}
where $\langle\alpha_i,\alpha_j\rangle$ denotes the \emph{Cartan integer}. 
\end{lem}

\begin{proof}
From \eqref{eq:3-2}, \eqref{eq:3-3}, \eqref{eq:3-4} we can compute 
\begin{equation} \label{eq:eq:3-5}
c_i^j=\frac {2(p_{s_i}(w_{\{\alpha_i,\alpha_j\}})-t)}{p_{s_j}(w_{\{\alpha_i,\alpha_j\}})}
\end{equation}
so it suffices to compute 
$p_{s_i}(w_{\{\alpha_i,\alpha_j\}})$ and 
$p_{s_j}(w_{\{\alpha_i,\alpha_j\}})$. 
In what follows we use the notation 
\[
a_{ij}:=\langle \alpha_i,\alpha_j\rangle \ \ (i\not=j), \quad a:=a_{ij}a_{ji}. 
\]
With this notation in place, note that by definition of the Cartan integers we have that the action of the simple reflections $s_j$ on the simple roots $\alpha_j$ may be expressed as 
\begin{equation} \label{eq:3-6}
s_j(\alpha_i)=\begin{cases} \alpha_i-a_{ij}\alpha_j \quad&(i\not=j),\\
                              -\alpha_i \quad&(i=j).
                              \end{cases}
\end{equation}

In order to prove the lemma, we consider each of the possible cases. 
                            
(i) In the case when the Dynkin diagram corresponding to $\{\alpha_i, \alpha_j\}$ is of the form 
\begin{picture}(35,10)                   
                       \put(5,5){\circle{5}}    
                       \put(7.5,5){\line(1,0){20}}
                       \put(30,5){\circle{5}}
                       
                       \put(3,-7){$i$} 
                       \put(27,-7){$j$}                                                                     
\end{picture}
the order of $s_is_j$ is $3$ so we have 
\[
w_{\{\alpha_i,\alpha_j\}}=s_is_js_i=s_js_is_j.
\]
Using 
Theorem~\ref{thmBilley} and~\eqref{eq:3-6} in this case we can compute that
\begin{align*}
&\sigma_{s_i}(w_{\{\alpha_i,\alpha_j\}})=\sigma_{s_i}(s_is_js_i)=\alpha_i+s_is_j(\alpha_i)=a\alpha_i-a_{ij}\alpha_j,\\
&\sigma_{s_j}(w_{\{\alpha_i,\alpha_j\}})=\sigma_{s_j}(s_js_is_j)=\alpha_j+s_js_i(\alpha_j)=a\alpha_j-a_{ji}\alpha_i. 
\end{align*}
Then~\eqref{eq:2-1} implies 
\begin{equation*}
p_{s_i}(w_{\{\alpha_i,\alpha_j\}})=(a-a_{ij})t,\quad p_{s_j}(w_{\{\alpha_i,\alpha_j\}})=(a-a_{ji})t.
\end{equation*}
Finally~\eqref{eq:eq:3-5} yields 
\begin{equation} \label{eq:3-7}
c_i^j= \frac{2(a-a_{ij}-1)}{(a-a_{ji})}
\end{equation}
and substituting $a=a_{ij}a_{ji}$, $a_{ij}=-1$ we obtain $c_i^j=-a_{ij}$ as desired.  

(ii) In the case 
\begin{picture}(35,10)                   
                       \put(5,5){\circle{5}}    
                       \put(7.5,6){\line(1,0){20}}
                       \put(7.5,4){\line(1,0){20}}
                       \put(30,5){\circle{5}}
                       
                       \put(3,-7){$i$} 
                       \put(27,-7){$j$}                                                                    
\end{picture}
the order of $s_is_j$ is $4$ so we have 
\[
w_{\{\alpha_i,\alpha_j\}}=s_is_js_is_j=s_js_is_js_i.
\]
Using the above together with 
Theorem~\ref{thmBilley} and~\eqref{eq:3-6} we may compute
\begin{align*}
&\sigma_{s_i}(w_{\{\alpha_i,\alpha_j\}})=\sigma_{s_i}(s_is_js_is_j)=\alpha_i+s_is_j(\alpha_i)=a\alpha_i-a_{ij}\alpha_j,\\
&\sigma_{s_j}(w_{\{\alpha_i,\alpha_j\}})=\sigma_{s_j}(s_js_is_js_i)=\alpha_j+s_js_i(\alpha_j)=a\alpha_j-a_{ji}\alpha_i  
\end{align*}
which is the same as case (i) above. Thus~\eqref{eq:3-7}
also holds in this case and since $a=a_{ij}a_{ji}=2$ we obtain $c_i^j=-a_{ij}$ as required. 

(iii) Finally, in the case 
\begin{picture}(35,10)                   
                       \put(5,5){\circle{5}}  
                       \put(7,6.5){\line(1,0){20}}
                       \put(7.5,5){\line(1,0){20}}
                       \put(7,3.5){\line(1,0){20}}
                       \put(30,5){\circle{5}} 
                                                                   
                       \put(3,-7){$i$} 
                       \put(27,-7){$j$}
                                                                    
\end{picture}
the element $s_is_j$ has order $6$ and thus 
\[
w_{\{\alpha_i,\alpha_j\}}=s_is_js_is_js_is_j=s_js_is_js_is_js_i.
\]
In this case we have $a=3$ so Theorem~\ref{thmBilley} and~\eqref{eq:3-6} yield that 
\begin{align*}
\sigma_{s_i}(w_{\{\alpha_i,\alpha_j\}})&=\sigma_{s_i}(s_is_js_is_js_is_j)
=\alpha_i+s_is_j(\alpha_i)+(s_is_j)^2(\alpha_i)=4\alpha_i-2a_{ij}\alpha_j,\\
\sigma_{s_j}(w_{\{\alpha_i,\alpha_j\}})&=\sigma_{s_j}(s_js_is_js_is_js_i)
=\alpha_j+s_js_i(\alpha_j)+(s_js_i)^2(\alpha_j)=4\alpha_j-2a_{ji}\alpha_i.
\end{align*}
Then from~\eqref{eq:2-1} we compute 
\begin{equation*}
p_{s_i}(w_{\{\alpha_i,\alpha_j\}})=(4-2a_{ij})t,\quad 
p_{s_j}(w_{\{\alpha_i,\alpha_j\}})=(4-2a_{ji})t.
\end{equation*}
Equation~\eqref{eq:eq:3-5} then implies 
\[
c_i^j=\frac{2(3-2a_{ij})}{4-2a_{ji}}
\]
and finally using that $a_{ij}a_{ji}=3$ we get that $c_i^j=-a_{ij}$ as desired.

This completes the proof of the lemma.
\end{proof}

From the above considerations we obtain the following proposition. 
\begin{prop} \label{Propqua} 
In the equivariant cohomology ring $H^*_S(Pet)$ of the Peterson variety, the following quadratic relations are satisfied:  
\begin{equation*}
\displaystyle\sum_{j=1}^{n}\langle\alpha_i,\alpha_j\rangle p_{s_i}p_{s_j}-2tp_{s_i}=0 \ \ (1\leq i\leq n).
\end{equation*}
\end{prop}

\begin{proof}
If $s_i$ and $s_j$ commute then $\langle \alpha_i,\alpha_j\rangle=0$, so by 
Lemma~\ref{lemcommute} the conclusion of 
Lemma~\ref{lemcartan} holds in this case. From this and \eqref{eq:3-4} we see that 
\eqref{eq:3-1} can be expressed as 
\begin{equation*}
p_{s_i}^2=t\cdot p_{s_i}-\frac{1}{2}\sum_{j\neq i} \langle\alpha_i,\alpha_j\rangle p_{s_i}p_{s_j}. 
\end{equation*}
Since $\langle \alpha_i,\alpha_i\rangle=2$
for any $i$, the above equation can be re-written to be of the form 
given in the statement of the proposition. 
\end{proof}

\section{The main theorem} \label{sect:4}

Let $(\langle\alpha_i,\alpha_j\rangle)_{1\leq i,j\leq n}$ be the Cartan matrix associated to a rank $n$ semisimple Lie algebra $\mathfrak g$. Using the coefficients in the Cartan matrix, we define an ideal $J$ in the polynomial ring 
$\mathbb{C}[x_1,\dots,x_n,t]$ as follows: 
\begin{align*}
J:=\left (\displaystyle\sum_{j=1}^{n}\langle\alpha_i,\alpha_j\rangle x_ix_j-2tx_i \mid 1\leq i\leq n \right).
\end{align*}
From Proposition~\ref{Propgene} and Proposition~\ref{Propqua} it then follows that the map sending $x_i$ to $p_{s_i}$ defines a  surjective $\mathbb{C}[t]$-algebra homomorphism 
\begin{equation}\label{eq:4-1}
\varphi: \mathbb{C}[x_1,\dots,x_n,t]/J\twoheadrightarrow H^*_S(Pet). 
\end{equation}
Here $H^*(BS)=\mathbb{C}[t]$ and $Pet$ denotes the Peterson variety 
associated to the Lie algebra $\mathfrak{g}$. 
Since 
$H^{odd}(Pet)=0$, as a $H^*(BS)$-module we have $H^*_S(Pet) \cong H^*(BS)\otimes H^*(Pet)$. 
Defining the ideal $\check{J}$ as 
\begin{equation} \label{eq:4-2}
\check{J}=\left( \displaystyle\sum_{j=1}^{n}\langle\alpha_i,\alpha_j\rangle x_ix_j \mid 1\leq i\leq n \right) 
\end{equation}
we then also have a surjective ring homomorphism 
\begin{equation} \label{eq:4-3}
\check\varphi\colon \mathbb{C}[x_1,\dots,x_n]/\check{J}\twoheadrightarrow H^*(Pet). 
\end{equation}
The following is the main theorem of this paper. 

\begin{thm}\label{mainthm}
The maps $\varphi$ and $\check{\varphi}$ of \eqref{eq:4-1} and \eqref{eq:4-3} are both isomorphisms. 
\end{thm}

In order to prove Theorem~\ref{mainthm} we use the theory of 
regular sequences. For reference we briefly recall the definition and 
a key property of regular sequences (cf. \cite{f-h-m}).

\begin{defn}\label{def:regular}
Let $R$ be a graded commutative algebra over $\QC$ and let $R_+$
denote the positive-degree elements in $R$.  Then a homogeneous
sequence $\f_1,\dots,\f_r \in R_+$ is a \emph{regular sequence} if 
$\f_k$ is a non-zero-divisor in the quotient ring $R/(\f_1,\dots,\f_{k-1})$ for every $1\le k\le r$.  This is equivalent to saying that $\f_1,\dots,\f_r$ is algebraically independent over $\QC$ and $R$ is a free $\QC[\f_1,\dots,\f_r]$-module.  
\end{defn}
It is a well-known fact (see for instance \cite[p.35]{stan96}) that 
a homogeneous sequence $\f_1,\dots,\f_r \in R_+$ 
is a regular sequence if and only if 
\begin{equation} \label{eq:regseq}
F(R/(\f_1,\dots,\f_r),s)=F(R,s)\prod_{k=1}^r(1-s^{\deg{\f_k}})
\end{equation}
where $F(R/(\f_1,\dots,\f_r),s)$ and $F(R,s)$ denote the Hilbert series of the graded rings $R/(\f_1,\dots,\f_r)$ and $R$, respectively.

The following proposition gives a convenient characterization of regular sequences. 

\begin{prop}\cite[Proposition 5.1]{f-h-m} \label{prop:zeroset} 
A sequence of positive-degree
  homogeneous elements $\f_1,\dots,\f_{\q}$ in the polynomial ring
  $\mathbb{C}[z_1,\dots,z_{\q}]$ is a regular sequence if and only if the
  solution set in $\mathbb{C}^{\q}$ of the equations $\f_1=0,\dots,\f_{\q}=0$
  consists only of the origin $\{0\}$. 
\end{prop}

We can now prove our main theorem. 

\begin{proof}[Proof of Theorem~\ref{mainthm}]
We first claim that if $\check{\varphi}$ is an isomorphism then it follows that $\varphi$ is an isomorphism.  To see this, suppose that $\check{\varphi}$ is an isomorphism. Then the sequence 
\[
\begin{split}
\f_i:&=\displaystyle\sum_{j=1}^{n}\langle\alpha_i,\alpha_j\rangle x_ix_j-2tx_i \quad \text{for $1\le i\le n$},\\
\f_{n+1}:&=t
\end{split}
\]
in $\QC[x_1,\dots,x_n,t]$ is regular, where $deg(x_i)=deg(t)=2$. Indeed, 
\[
\begin{split}
&F(\QC[x_1,\dots,x_n,t]/(\f_1,\dots,\f_n,\f_{n+1}),s)\\
=&F(\QC[x_1,\dots,x_n]/\check{J},s)\\
=&(1+s^2)^n\\
=&\frac{1}{(1-s^2)^{n+1}}\cdot(1-s^{4})^n(1-s^{2})\\
=&F(\QC[x_1,\dots,x_n,t],s)\prod_{i=1}^{n+1}(1-s^{\deg{\f_i}})
\end{split}
\]
so this follows from~\eqref{eq:regseq}. Note that  
a subsequence $\f_1,\dots,\f_n$ of a regular sequence $\f_1,\dots,\f_{n+1}$ is again a regular sequence, so 
from~\eqref{eq:regseq} and~\eqref{eq:2-3} we obtain 
\[
\begin{split}
F(\QC[x_1,\dots,x_n,t]/J,s)&=F(\QC[x_1,\dots,x_n,t]/(\f_1,\dots,\f_n),s)\\
&=\frac{1}{(1-s^2)^{n+1}}\prod_{i=1}^n(1-s^{\deg \f_i})\\
&=\frac{(1+s^2)^n}{1-s^2}\\
&=F(H^*_S(Pet),s)
\end{split}
\]
from which it follows that $\varphi$ is an isomorphism. 

Thus it suffices to check that $\check{\varphi}$ is an isomorphism. 
We already know that $\check{\varphi}$ is surjective and from equation~\eqref{eq:2-3} we know that $F(H^*(Pet),s)=(1+s^2)^n$. 
Thus in order to show that $\check{\varphi}$ is injective it suffices to show that 
\begin{equation} \label{eq:4-4}
F(\mathbb{C}[x_1,\dots,x_n]/\check{J},s)=(1+s^2)^n. 
\end{equation}
Note that by \eqref{eq:regseq}, the equality \eqref{eq:4-4} is equivalent to the statement that $\sum_{j=1}^{n}\langle\alpha_i,\alpha_j\rangle x_ix_j$ \quad $(1\le i\le n)$ is a regular sequence.
Furthermore, by Proposition~\ref{prop:zeroset} , in order to prove 
\eqref{eq:4-4} it in turn suffices to show 
that the zero set of the collection of quadratic equations 
\begin{equation} \label{eq:4-5}
\sum_{j=1}^{n}\langle\alpha_i,\alpha_j\rangle x_ix_j=0 \quad (1\le i\le n),
\end{equation}
given by the generators of the ideal $\check{J}$ of~\eqref{eq:4-2} is $\{0\}$, i.e., the equations~\eqref{eq:4-5}  have only the trivial solution.

Suppose in order to derive a contradiction that 
\eqref{eq:4-5} has a non-trivial solution $(b_1,\dots,b_n)$. In particular, setting $I=\{i\mid b_i\not=0\}$, we have $I \neq \emptyset$
and so since 
$b_i\not=0$ for $i\in I$ we obtain from~\eqref{eq:4-5} that 
\[
\sum_{j\in I}\langle\alpha_i,\alpha_j\rangle b_j=0 \quad (i\in I). 
\]
Since 
$(\langle\alpha_i,\alpha_j\rangle)_{i,j\in I}$
is a $\lvert I \rvert \times \lvert I \rvert$ square matrix which is again the Cartan matrix of a semisimple Lie algebra, it must be positive definite 
\cite[section 2.4]{hu97} and in particular non-singular. Thus the $b_i$ must be $0$ for $i \in I$, contradicting the assumption on $I$. Thus~\eqref{eq:4-5} has only the trivial solution, as desired. 
\end{proof}

\begin{rmk}
Theorem~\ref{mainthm} is a generalization to all Lie types of the 
computation given in \cite{f-h-m}. Indeed, the generators of the ideal given in \cite{f-h-m} are the same as those given above, up to a scalar factor of $1/2$. 
\end{rmk}

\begin{rmk}
In fact, Theorem~\ref{mainthm} holds also with $\mathbb{Q}$ coefficients. Indeed, since both 
$\varphi$ and $\check{\varphi}$ can be defined over $\mathbb{Z}$, if the maps become isomorphisms upon tensoring with 
$\mathbb{C}$ then they are also isomorphisms upon tensoring with $\mathbb{Q}$. 
\end{rmk}

\end{document}